\def\version{29/07/2013 version 6 (to appear in
\emph{Journal of Homotopy and Related Structures})
\hfill
\href{http://arxiv.org/abs/1204.4878}{arXiv:1204.4878}}

\documentclass[11pt,a4paper]{amsart}
\usepackage{fullpage}

\usepackage{amsmath,amssymb,amscd,amsxtra,upref}

\usepackage{fixltx2e}

\usepackage[numeric,lite]{amsrefs}
\usepackage{hyperref}
\usepackage[all]{xy}
\usepackage{mathrsfs}
\usepackage{stmaryrd}

\usepackage{pigpen}
\def\PO{\text{\pigpenfont R}}
\def\PB{\text{\pigpenfont J}}

\theoremstyle{plain}
\newtheorem{thm}{Theorem}[section]
\newtheorem{lem}[thm]{Lemma}
\newtheorem{prop}[thm]{Proposition}
\newtheorem{cor}[thm]{Corollary}

\theoremstyle{definition}
\newtheorem{rem}[thm]{Remark}

\numberwithin{equation}{section}

\def\ie{i.e.}
\def\ds{\displaystyle}
\def\:{\colon}
\def\.{\cdot}

\def\<{\left\langle}
\def\>{\right\rangle}
\def\({\left(}
\def\){\right)}
\def\ph#1{\phantom{#1}}
\def\epsilon{\varepsilon}
\def\leq{\leqslant}
\def\geq{\geqslant}

\def\lra{\longrightarrow}
\def\Lra{\Longrightarrow}
\def\LRA{{\ \Lra\ }}

\def\tilde#1{\widetilde{#1}}
\def\iso{\cong}
\def\emptyset{\varnothing}

\DeclareMathOperator{\im}{im}

\def\F{\mathbb{F}}

\def\Q{\mathbb{Q}}
\def\N{\mathbb{N}}

\def\Z{\mathbb{Z}}

\DeclareMathOperator{\Ext}{Ext}

\DeclareMathOperator{\inc}{inc}
\DeclareMathOperator{\Tor}{Tor}
\DeclareMathOperator{\tors}{tors}

\DeclareMathOperator*{\colim}{colim}
\DeclareMathOperator*{\hocolim}{hocolim}

\def\phi{\varphi}

\DeclareMathOperator{\TP}{TP}

\DeclareMathOperator{\exc}{exc}

\def\Kriz{K\v{r}\'{\i}\v{z}}
\newdir{ >}{{}*!/-8pt/@{>}}

\begin{document}

\title[$BP$: Close encounters of the $E_\infty$ kind]
{\boldmath{$BP$}\,: Close encounters of the $E_\infty$ kind}
\author{Andrew Baker}
\email{a.baker@maths.gla.ac.uk}
\address{
School of Mathematics \& Statistics,
University of Glasgow,
Glasgow G12 8QW,
Scotland.}
\urladdr{http://www.maths.gla.ac.uk/$\sim$ajb}
\thanks{The author would like to thank Bob Bruner, Mike
Mandell, Peter May, Birgit Richter, John Rognes and Markus
Szymik; special thanks are due to Tyler Lawson who pointed
out the usefulness of coning off Moore spectra rather than
spheres. Finally, we note that this paper would not exist
without the inspiration provided by Stewart Priddy's elegant
cellular construction of the Brown-Peterson spectrum. \\
This research was supported by funding from RCUK}

\begin{abstract}
Inspired by Stewart Priddy's cellular model for the $p$-local
Brown-Peterson spectrum $BP$, we give a construction of
a $p$-local $E_\infty$ ring spectrum $R$ which is a close
approximation to $BP$. Indeed we can show that if $BP$
admits an $E_\infty$ structure then these are weakly
equivalent as $E_\infty$ ring spectra. Our inductive
cellular construction makes use of power operations on
homotopy groups to  define homotopy classes which are
then killed by attaching $E_\infty$ cells.
\end{abstract}

\date{\version}

\maketitle

\section*{Introduction}

The notion of an \emph{$E_\infty$ ring spectrum} arose
in the 1970s, and was studied in depth by Peter May
\emph{et al} in \cite{LNM533}, then later reinterpreted
in the framework of~\cite{EKMM} as equivalent to that
of a \emph{commutative $S$-algebra}. A great deal of
work on the existence of $E_\infty$ structures using
various obstruction theories has led to a considerable
enlargement of our range of known examples. A useful
recent discussion of relationships between various
aspects of these topics can be found in~\cite{JPM:Einfty?}.

However, despite this, there are some gaps in our
knowledge. The question that is a major motivation
of this paper is
\begin{itemize}
\item
\emph{Does the $p$-local Brown-Peterson spectrum~$BP$
for a prime~$p$ admit an $E_\infty$ ring structure?}
\end{itemize}
This has been flagged up as an outstanding problem for
almost four decades, despite various attempts to answer
it.

Around 1980, Stewart Priddy~\cite{SP:BP} showed how to
build an efficient cellular model for the spectrum $BP$.
This stimulated the later work of~\cite{HKM} (where the
basic method was analysed and extended to $E_\infty$ ring
spectra), then~\cite{AJB&JPM} (where outstanding issues
about the spectrum case were addressed) and~\cite{TAQ}
(where the analogous multiplicative theory was described
using topological Andr\'e-Quillen homology in place of
ordinary homology). However, none of this answers the
above question!

Some other recent results also add to the uncertainty.
Niles Johnson and Justin Noel~\cite{NJ&JN} have shown
that for some small primes at least, the natural
orientation map of ring spectra $MU\lra BP$ cannot be
$E_\infty$ (or even $H_\infty$). On the other hand,
Mike Hill, Tyler Lawson and Niko Naumann~\cites{MH&TL,TL&NN}
have shown that for the primes~$2$ and~$3$, $BP\<2\>$
admits an $E_\infty$ ring structure. Finally, partial
results on higher coherence of the multiplication on~$BP$
have been proved by Birgit Richter~\cite{BR:BP-En}, and
Maria Basterra and Mike Mandell~\cite{MB&MM:BP} (the
latter uses ideas pioneered in an influential but
unpublished preprint of Igor \Kriz~\cite{IK:BP}).

Our main purpose in this paper is to give a prescription
for constructing a close approach to $BP$ at a prime~$p$.
We will show that there is a connective finite type
$p$-local $E_\infty$ ring spectrum $R$ such that the
following hold.
\begin{itemize}
\item
The homotopy $\pi_*R$ is torsion-free.
\item
There is a morphism of ring spectra $BP\lra R$ which
is a rational weak equivalence.
\item
If $BP$ admits an $E_\infty$ ring structure then there
is a weak equivalence of $E_\infty$ ring spectra
$R\xrightarrow{\;\sim\;}BP$.
\end{itemize}

Our construction proceeds in two main stages, the first
of which yields a morphism of $p$-local $E_\infty$ ring
spectra $R_\infty\lra MU_{(p)}$ so that the composition
\[
R_\infty\xrightarrow{\ph{\;\epsilon\;}} MU_{(p)}
        \xrightarrow{\;\epsilon\;} BP
\]
with the Quillen projection $\epsilon$ is a morphism
of ring spectra which induces an epimorphism on
$\pi_*(-)$ and is a rational equivalence. The second
stage gives a morphism of $E_\infty$ ring spectra
$R_\infty\lra R$ which is a rational equivalence
and where $\pi_*(R)$ is torsion free. One source of
difficulty with our construction is that if $R$ is
an $E_\infty$ realisation of $BP$, then there can
be no $E_\infty$ morphism $R\lra MU_{(p)}$
by~\cite{HKM}*{theorem~2.11}. If we could produce
any map of spectra $R\lra BP$ which is an equivalence
on the bottom cell then the composition $BP\lra R\lra BP$
would be a weak equivalence and so would each of the
maps $BP\lra R$ and $R\lra BP$.

\section{Attaching $E_\infty$ cells to commutative
$S$-algebras}\label{sec:Cells}

We recall the idea of attaching $E_\infty$ cells to a
commutative $S$-algebra. Details can be found in~\cite{EKMM},
and it was exploited in~\cite{TAQ} to describe topological
Andr\'e-Quillen homology of CW commutative $S$-algebras.
We will make use of various obstructions involving free
commutative $S$-algebras. Recall from~\cite{EKMM} that if
$X$ is an $S$-module then the free commutative $S$-algebra
on $X$ is
\[
\mathbb{P}X = \mathbb{P}_SX = \bigvee_{r\geq0} X^{(r)}/\Sigma_r.
\]
When $X$ is cofibrant, for each $r\geq1$ the natural projection
provides a weak equivalence
\begin{equation}\label{eq:ExtPow-wkeqce}
E\Sigma_r\ltimes_{\Sigma_r}X^{(r)}
                    \xrightarrow{\;\sim\;} X^{(r)}/\Sigma_r.
\end{equation}

Let $E$ be a commutative $S$-algebra and let $f\:\bigvee_i S^n\lra E$
be a map from a finite wedge of $n$-spheres. Then there
is a unique extension of $f$ to a morphism of commutative
$S$-algebras $\tilde f\:\mathbb{P}(\bigvee_i S^n)\lra E$
from the free commutative $S$-algebra on $\bigvee_i S^n$.
Then the pushout diagram of commutative $S$-algebras
\[
\xymatrix{
\mathbb{P}(\bigvee_i S^n) \ar[rr]^{\ph{abc}\tilde f}
    \ar[dd]_{\mathbb{P}(\inc)}\ar@{}[ddrr]|{\PO}
    %{\text{\huge$\ulcorner$}}
        && E\ar[dd] \\
          &&& \\
\mathbb{P}(\bigvee_i D^{n+1}) \ar[rr] && E/\!/f
}
\]
defines $E/\!/f$ which we can regard as obtained from~$E$
by attaching $E_\infty$ cells. In fact, we can take
\[
E/\!/f =
\mathbb{P}(\bigvee_i D^{n+1})\wedge_{\mathbb{P}(\bigvee_i S^n)}E
\]
where $\mathbb{P}(\bigvee_i D^{n+1})$ and $E$ are
$\mathbb{P}(\bigvee_i S^n)$-algebras in the evident way.

The homology of extended powers has been well studied
and we can deduce the following.
\begin{prop}\label{prop:PSn-rational}
For $n\in\N$, we have
\[
H_*(\mathbb{P}S^{2n-1};\Q) = \Lambda_\Q(x_{2n-1}),
\quad
H_*(\mathbb{P}S^{2n};\Q) = \Q[x_{2n}],
\]
where $x_m\in H_m(\mathbb{P}S^m;\Q)$ is the image of the
homology generator of $H_m(S^m;\Q)$.
\end{prop}
\begin{proof}
The weak equivalences of~\eqref{eq:ExtPow-wkeqce} combine
to give a weak equivalence
\[
\bigvee_{r\geq0} E\Sigma_r\ltimes_{\Sigma_r}(S^{2n-1})^{(r)}
\xrightarrow{\;\sim\;}
\bigvee_{r\geq0}(S^{2n-1})^{(r)}/\Sigma_r
                                       = \mathbb{P}S^{2n-1}.
\]
By~\cite{LNM1213}*{chapter~VIII}, for $r\geq2$ we have
\[
H_*(E\Sigma_r\ltimes_{\Sigma_r}(S^{2n-1})^{(r)};\Q)
            = H_*((S^{2n-1})^{(r)};\Q)_{\Sigma_r} = 0,
\]
since the permutation action of $\Sigma_r$ on the factors
is equivalent to the sign representation,
\[
H_*((S^{2n-1})^{(r)};\Q) \iso \Q^-,
\]
which is a summand of the regular representation $\Q[\Sigma_r]$,
hence it has trivial cohomology, and in particular trivial
coinvariants. Thus we have
\[
H_*(\mathbb{P}S^{2n-1};\Q) = \Lambda_\Q(x_{2n-1}),
\]
where
\[
x_{2n-1}\in H_{2n-1}(D_1S^{2n-1};\Q) \iso
 H_{2n-1}(E\Sigma_r\ltimes_{\Sigma_r}(S^{2n-1})^{(r)};\Q).
\]

Similarly,
\[
\mathbb{P}S^{2n} =\bigvee_{r\geq0}D_rS^{2n} \sim
\bigvee_{r\geq0} E\Sigma_r\ltimes_{\Sigma_r}(S^{2n})^{(r)},
\]
but this time the $\Sigma_r$ action on the factors is trivial
giving
\[
H_*((S^{2n-1})^{(r)};\Q) \iso \Q,
\]
hence
\[
H_*(E\Sigma_r\ltimes_{\Sigma_r}(S^{2n})^{(r)};\Q)
         = H_*((S^{2n})^{(r)};\Q)_{\Sigma_r} = \Q
\]
concentrated in degree $2nr$. It follows easily that
$H_*(\mathbb{P}S^{2n};\Q)$ is polynomial on the stated
generator.
\end{proof}

The next result is fundamental,
see~\cites{LNM533,JPM:HomOps,JPM:Hinfty,NJK:Transfer,NJK&JMcC:HF2InfLoopSpcs}.
Here we use the convention that the excess of the
empty exponent sequence is $\exc(\emptyset) = \infty$.
\begin{thm}\label{thm:H*PX}
If $X$ is connective then for a prime $p$, $H_*(\mathbb{P}X;\F_p)$
is the free commutative graded $\F_p$-algebra generated by elements
$Q^Ix_j$, where $x_j$ for $j\in J$ gives a basis for $H_*(X;\F_p)$,
and $I=(\epsilon_1,i_1,\epsilon_2,\ldots,\epsilon_\ell,i_\ell)$
is admissible with $\exc(I)+\epsilon_1>|x_j|$ when~$p$ is odd,
while $I=(i_1,\ldots,i_\ell)$ is admissible with $\exc(I)>|x_j|$
when~$p=2$.
\end{thm}

Using the notation $R\<G\>$ for the free commutative graded
algebra over $R$ on a collection of homogeneous generators~$G$,
this gives the following formulae. Thus for $p$ odd,
\[
H_*(\mathbb{P}X;\F_p) =
 \F_p\<Q^Ix_j:\text{$j\in J$, $\exc(I)+\epsilon_1>|x_j|$}\;\>
\]
is polynomial on the stated generators with $|Q^Ix_j|$ even
and exterior on those generators with $|Q^Ix_j|$ odd, while
for $p=2$,
\begin{align*}
H_*(\mathbb{P}X;\F_2)
  &= \F_2\<Q^Ix_j:\text{$j\in J$, $\exc(I)>|x_j|$}\;\> \\
  &= \F_2[Q^Ix_j:\text{$j\in J$, $\exc(I)>|x_j|$}\;].
\end{align*}
\begin{rem}\label{rem:triviality}
If $E$ is a commutative $S$-algebra and that $f\:X\lra E$
is a map of spectra for which the induced homomorphism
\[
f_*\:H_*(X;\F_p)\lra H_*(E;\F_p)
\]
is trivial. Then by Theorem~\ref{thm:H*PX}, the induced
ring homomorphism
\[
\tilde{f}_*\:H_*(\mathbb{P}X;\F_p)\lra H_*(E;\F_p)
\]
is also trivial since it is a homomorphism of algebras over
the  Dyer-Lashof algebra.
\end{rem}

We record some results on the attaching of $E_\infty$ cones
to commutative $S$-algebras and its effect on ordinary
homology. We will make repeated use of the K\"unneth spectral
sequence of~\cite{EKMM}. By~\cites{AB&AL:ASS} this is
multiplicative, and for a prime~$p$ an extension of the work
of~\cites{TH:THH,MB&MM:BP} shows that it has Dyer-Lashof
operations.

First we give some easy observations on rational homology.
\begin{prop}\label{prop:rational-odd}
Suppose that $E$ is a connective commutative $S$-algebra
and let $n\in\N$. \\
\emph{(a)} If $f\:S^{2n-1}\lra E$ is a map for which the
induced homomorphism $f_*\:H_*(S^{2n-1};\Q)\lra H_*(E;\Q)$
is trivial, then
\[
H_*(E/\!/f;\Q) = H_*(E;\Q)[w],
\]
where $w\in H_{2n}(E/\!/f;\Q)$. \\
\emph{(b)} If $f\:S^{2n}\lra E$ is a map for which the
induced homomorphism $f_*\:H_*(S^{2n};\Q)\lra H_*(E;\Q)$
is trivial, then
\[
H_*(E/\!/f;\Q) = \Lambda_{H_*(E;\Q)}[z],
\]
where $z\in H_{2n+1}(E/\!/f;\Q)$.
\end{prop}
\begin{proof}
Recall Proposition~\ref{prop:PSn-rational}.

\noindent
(a) There is a multiplicative K\"unneth spectral
sequence~\cites{EKMM,AB&AL:ASS} of form
\[
\mathrm{E}^2_{s,t} =
  \Tor^{H_*(\mathbb{P}S^{2n-1};\Q)}_{s,t}(\Q,H_*(E;\Q))
                    \LRA H_{s+t}(E/\!/f;\Q),
\]
where we have
\[
\mathrm{E}^2_{*,*} = \Gamma_{H_*(E;\Q)}(w) = H_*(E;\Q)[w],
\]
with $w\in\mathrm{E}^2_{1,2n-1}$. As $w$ is an infinite
cycle for degree reasons, the result follows.

\noindent
(b) Here the relevant K\"unneth spectral sequence
\[
\mathrm{E}^2_{s,t} =
  \Tor^{H_*(\mathbb{P}S^{2n};\Q)}_{s,t}(\Q,H_*(E;\Q))
                    \LRA H_{s+t}(E/\!/f;\Q),
\]
has
\[
\mathrm{E}^2_{*,*} = \Lambda_{H_*(E;\Q)}(z)
\]
with $z\in\mathrm{E}^2_{1,2n}$ which is an infinite cycle for
degree reasons.
\end{proof}

Of course we can replace a single sphere by a wedge of spheres
in this result.

Now we will describe analogous results in positive characteristic.
When the context makes this unambiguous, we will often write
$H_*(-)$ for $H_*(-;\F_p)$ and $\otimes$ for $\otimes_{\F_p}$.
In the following, a \emph{$p$-truncated algebra} will mean
a quotient $\F_p$-algebra of the form
\[
\TP_p(x) = \F_p[x]/(x^p).
\]
and we will denote this by $\TP(x)$ when the prime~$p$
is clear. It is standard that a divided power algebra
on an element~$x$,
\[
\Gamma_{\F_p}(x) = \F_p\{1,\gamma_1(x),\gamma_2(x),\ldots\},
\]
is a tensor product of $p$-truncated algebras:
\[
\Gamma_{\F_p}(x) = \bigotimes_{r\geq0}\TP(\gamma_{p^r}(x)).
\]
Here the product is given by
\[
\gamma_r(x)\gamma_s(x) = \binom{r+s}{r}\gamma_{r+s}(x)
\]
and so for every $r$,
\[
\gamma_r(x)^p = 0.
\]
Furthermore, if $r$ has $p$-adic expension
\[
r = r_0 + r_1 p + \cdots + r_\ell p^\ell,
\]
where $r_i=0,1,\ldots,p-1$, then there is a non-zero
element $c_r\in\F_p$ for which
\[
\gamma_r(x) =
c_r\gamma_{1}(x)^{r_0}\gamma_{p}(x)^{r_1}
             \cdots\gamma_{p^\ell}(x)^{r_\ell}.
\]

To prove the odd primary case in our next result,
we make use of work of Hunter~\cite{TH:THH}.
\begin{prop}\label{prop:Hunter}
Let $p$ be an odd prime and let $E$ be a connective
commutative $S$-algebra. Suppose that $n\in\N$ and
$f\:S^{2n-1}\lra E$ is a map for which the induced
homomorphism $f_*\:H_*(S^{2n-1};\F_p)\lra H_*(E;\F_p)$
is trivial. Then the K\"unneth spectral sequence
\[
\mathrm{E}^2_{s,t} =
\Tor^{H_*(\mathbb{P}S^{2n-1};\F_p)}_{s,t}(\F_p,H_*(E;\F_p))
   \Lra H_{s+t}(E/\!/f;\F_p)
\]
has the following properties. \\
\emph{(a)}
The homology of\/ $\mathbb{P}S^{2n-1}$ is the free commutative
graded algebra
\begin{equation}\label{eq:H*PS(odd)}
H_*(\mathbb{P}S^{2n-1};\F_p) =
   \F_p\<Q^Ix_{2n-1}:\exc(I)+\epsilon_1>2n-1\;\>.
\end{equation}
\emph{(b)}
The\/ $\mathrm{E}^2$-term is a tensor product
\[
\mathrm{E}^2_{*,*} =
H_*(E;\F_p)\otimes \mathcal{D}\otimes \mathcal{E},
\]
of subalgebras, where $\mathcal{D}$ and $\mathcal{E}$ have the
following descriptions:
\begin{itemize}
\item
$\mathcal{D}$ is a tensor product of infinitely many divided power
algebras each having the form $\Gamma_{\F_p}([Q^Ix_{2n-1}])$
with a generator $[Q^Ix_{2n-1}]\in\mathrm{E}^2_{1,|Q^Ix_{2n-1}|}$
for each odd degree exterior generator occurring
 in~\eqref{eq:H*PS(odd)};
\item
$\mathcal{E}$ is an exterior algebra with a generator
$[Q^Ix_{2n-1}]\in\mathrm{E}^2_{1,|Q^Ix_{2n-1}|}$ for
each even degree polynomial generator $Q^Ix_{2n-1}$
occurring  in~\eqref{eq:H*PS(odd)}.
\end{itemize}
\emph{(c)}
In the above spectral sequence,
\[
\mathrm{E}^2_{*,*} =\cdots=\mathrm{E}^{p-1}_{*,*},
\quad
\mathrm{E}^{p}_{*,*}=\mathrm{E}^\infty_{*,*},
\]
where the differential $d^{p-1}$ acts on the divided
power generators of $\mathcal{D}$ by
\[
d^{p-1}\gamma_{p^r}([Q^Ix_{2n-1}]) \doteq
\begin{cases}
[\beta Q^{(|Q^Ix_{2n-1}|+1)/2}Q^Ix_{2n-1}]\gamma_{p^r-p}([Q^Ix_{2n-1}])
                   & \text{\rm if $r\geq1$}, \\
\qquad 0           & \text{\rm if $r=0$},
\end{cases}
\]
where $\doteq$ means `equal up to multiplication by a unit
in $\F_p$'.
\end{prop}
\begin{proof}
Using a standard Koszul resolution over the free algebra
$H_*(\mathbb{P}S^{2n-1})$ we obtain the stated form for
the $\mathrm{E}^2$-term. The statement about the
differentials involves a suitable reinterpretation
of~\cite{TH:THH}*{proposition~11} together with the
multiplicative structure of the spectral sequence.
\end{proof}

%The situation for $p=2$ is simpler to describe.
%\begin{prop}\label{prop:Hunter(2)}
%Let $p=2$ and let $E$ be a connective commutative $S$-algebra.
%Suppose that $n\in\N$ and $f\:S^{2n-1}\lra E$ is a map whose
%induced homomorphism $f_*\:H_*(S^{2n-1};\F_2)\lra H_*(E;\F_2)$
%is trivial. Then the K\"unneth spectral sequence
%\[
%\mathrm{E}^2_{s,t} =
%\Tor^{H_*(\mathbb{P}S^{2n-1};\F_2)}_{s,t}(\F_2,H_*(E;\F_2))
%   \Lra H_{s+t}(E/\!/f;\F_2)
%\]
%has the following properties. \\
%\emph{(a)}
%The homology of\/ $\mathbb{P}S^{2n-1}$ is the graded
%polynomial algebra
%\[
%H_*(\mathbb{P}S^{2n-1};\F_2) = \F_2[Q^Ix_{2n-1}:\exc(I)>2n-1\;].
%\]
%\emph{(b)}
%The\/ $\mathrm{E}^2$-term is the exterior algebra
%\[
%\mathrm{E}^2_{*,*} = \Lambda_{H_*(E;\F_2)}([Q^Ix_{2n-1}]:\exc(I)>2n-1),
%\]
%with a generator $[Q^Ix_{2n-1}]\in\mathrm{E}^2_{1,|Q^Ix_{2n-1}|}$
%for each polynomial generator $Q^Ix_{2n-1}$ in \emph{(a)}. \\
%\emph{(c)}
%This spectral sequence has trivial differentials from
%$\mathrm{E}^2$ onwards.
%\end{prop}
%\begin{proof}
%Using a standard Koszul resolution over the free algebra
%$H_*(\mathbb{P}S^{2n-1})$ we obtain the stated form for
%the $\mathrm{E}^2$-term. The exterior generators all lie
%in $\mathrm{E}^2_{1,*}$ and so must be infinite cycles.
%The multiplicative structure of the spectral sequence
%shows that all differentials are trivial.
%\end{proof}

The situation for $p=2$ is simpler to describe and we state
it in greater generality than we actually need for the present
work.
\begin{prop}\label{prop:Hunter(2)}
Let $p=2$ and let $E$ be a connective commutative $S$-algebra.
Suppose that $0\leq n\in\Z$ and $f\:S^n\lra E$ is a map whose
induced homomorphism $f_*\:H_*(S^n;\F_2)\lra H_*(E;\F_2)$ is
trivial. Then the K\"unneth spectral sequence
\[
\mathrm{E}^2_{s,t} =
\Tor^{H_*(\mathbb{P}S^n;\F_2)}_{s,t}(\F_2,H_*(E;\F_2))
               \Lra H_{s+t}(E/\!/f;\F_2)
\]
has the following properties. \\
\emph{(a)}
The homology of\/ $\mathbb{P}S^n$ is the graded polynomial
algebra
\[
H_*(\mathbb{P}S^n;\F_2) = \F_2[Q^Ix_n:\exc(I)>n\;].
\]
\emph{(b)}
The\/ $\mathrm{E}^2$-term is the exterior algebra
\[
\mathrm{E}^2_{*,*} = \Lambda_{H_*(E;\F_2)}([Q^Ix_n]:\exc(I)>n),
\]
with a generator $[Q^Ix_n]\in\mathrm{E}^2_{1,|Q^Ix_n|}$
for each polynomial generator $Q^Ix_n$ in \emph{(a)}. \\
\emph{(c)}
This spectral sequence has trivial differentials from
$\mathrm{E}^2$ onwards.
\end{prop}
\begin{proof}
Using a standard Koszul resolution over the free algebra
$H_*(\mathbb{P}S^n)$ we obtain the stated form for the
$\mathrm{E}^2$-term. The exterior generators all lie
in $\mathrm{E}^2_{1,*}$ and so must be infinite cycles.
The multiplicative structure of the spectral sequence
shows that all differentials are trivial.
\end{proof}

\begin{thm}\label{thm:Cone-homology}
Let $p$ be a prime. Suppose that $E$ is a connective commutative
$S$-algebra and that $f\:S^{2n-1}\lra E$ is a map for which the
induced homomorphism $f_*\:H_*(S^{2n-1};\F_p)\lra H_*(E;\F_p)$
is trivial. Then there is an element $u\in H_{2n}(E/\!/f;\F_p)$
such that if $p$ is odd,
\[
H_*(E/\!/f;\F_p) =
H_*(E;\F_p)\<Q^Iu:\exc(I)+\epsilon_1>2n,
  \;\text{\rm with $i_1=0$ if $\epsilon_1=1$ and $|Q^Iu|$ odd}\>,
\]
while if $p=2$,
\[
H_*(E/\!/f;\F_2) = H_*(E;\F_2)[Q^Iu:\exc(I)>2n].
\]
\end{thm}
\begin{proof}
Taking into account the results of
Propositions~\ref{prop:Hunter} and~\ref{prop:Hunter(2)},
we find that the $\mathrm{E}^\infty$-term is a tensor
product of algebras
\[
\mathrm{E}^\infty_{*,*}
    = H_*(E)\otimes \mathcal{D}'\otimes\mathcal{E}',
\]
where
\begin{itemize}
\item
$\mathcal{D}'$ is a tensor product of infinitely many $p$-truncated
algebras each having  the form $\TP([Q^Ix_{2n-1}])$ with a generator
$[Q^Ix_{2n-1}]\in\mathrm{E}^\infty_{1,|Q^Ix_{2n-1}|}$ for each
exterior generator in~\eqref{eq:H*PS(odd)},
\item
$\mathcal{E}'$ is an exterior algebra with a generator
$[Q^Ix_{2n-1}]\in\mathrm{E}^\infty_{1,|Q^Ix_{2n-1}|}$
for each polynomial generator $Q^Ix_{2n-1}$ listed
in~\eqref{eq:H*PS(odd)}.
\end{itemize}

When $p$ is odd, each exterior generator $Q^Ix_{2n-1}$
is of odd degree so it gives rise to a $p$-truncated
algebra $\TP([Q^Ix_{2n-1}])$ concentrated in even degrees.
When $p=2$, $\mathcal{D}'$ is trivial and $\mathcal{E}'$
is generated by elements $[Q^Ix_{2n-1}]$ satisfying
$[Q^Ix_{2n-1}]^2=0$. In each case we need to show these
generators represent elements which are not nilpotent
in $H_*(E/\!/\gamma)$. We do this using Dyer-Lashof
operations, using a well known argument, see for
example~\cite{HL&IM}.

If $p$ is an odd prime, set
\[
k = \frac{|Q^Ix_{2n-1}|+1}{2},
\]
and if $p=2$, take
\[
k = |Q^Ix_{2n-1}|+1.
\]
Then in the $\mathrm{E}^2$-term we have
\[
Q^k[Q^Ix_{2n-1}] = [Q^kQ^Ix_{2n-1}] \neq 0
\]
since $Q^kQ^I$ is admissible. This shows that in $H_*(E/\!/\gamma)$
$[Q^Ix_{2n-1}]$ represents an element whose $p$-th power is represented
by $[Q^kQ^Ix_{2n-1}]$, thus resolving the multiplicative extensions
in the filtration.
\end{proof}

\begin{rem}\label{rem:Cone-homology-p=2}
When $p=2$ this results applies to all the cases of
Proposition~\ref{prop:Hunter(2)}. Thus for $f\:S^n\lra E$ with
$n\geq0$ we have
\[
H_*(E/\!/f;\F_2) = H_*(E;\F_2)[Q^Iu:\exc(I)>n+1].
\]
\end{rem}

\section{Power operations for $E_\infty$ ring spectra}\label{sec:PowOps}

We refer to \cite{LNM1176} for work on power operations, in particular
Bruner's chapters~IV and~V. Our main use of this is in connection with
applying `the first operation above the $p$-th power' $\beta\mathcal{P}^{k+1}$
to give a homotopy element of degree~$2k$. Here are the results we will
use.

At the prime $2$, we have
\begin{thm}\label{thm:PowOps2}
Suppose that $E$ is a connective $2$-local $E_\infty$ ring spectrum
for which $0=\eta1\in\pi_1E$. Then for\/ $r\geq1$, the operation
$\mathcal{P}^{2^{r+1}-1}$ is defined on $\pi_{2^{r+1}-2}E$, giving
a map
\[
\mathcal{P}^{2^{r+1}-1}\:\pi_{2^{r+1}-2}E \lra \pi_{2^{r+2}-3}E.
\]
Moreover, the indeterminacy is trivial and the operation
$2\mathcal{P}^{2^{r+1}-1}$ is trivial.
\end{thm}
\begin{proof}
We will write $n=2^{r+1}-2$.

Applying~\cite{LNM1176}*{proposition~V.1.5} to the skeleton
$D_2^1S^{2n}$, we have
\[
i=1,\quad j=n+1,\quad \phi(i) = 1,
\]
and so
\[
n \equiv -2\mod{(2)},
\]
hence the operation $\mathcal{P}^{2^{r+1}-1}$ is defined
on $\pi_{2^{r+1}-2}E$. Also, by~\cite{LNM1176}*{theorem~V.1.8}
we have with $j=a=b=0$ and $w\in\pi_{2^{r+1}-2}E$,
\[
2\mathcal{P}^{2^{r+1}-1}w = 0
\]
since by assumption the natural map $\pi_1S\lra\pi_1E$
is trivial. Similarly, since $n\equiv 2\bmod{(4)}$, the
indeterminacy is trivial by~\cite{LNM1176}*{table~V.1.3}.
\end{proof}

For odd primes we have
\begin{thm}\label{thm:PowOps}
Let $p$ be an odd prime. Suppose that $E$ is a connective
$p$-local $E_\infty$ ring spectrum for which $0=\alpha_1 1\in\pi_{2p-3}E$.  \\
Then for\/ $r\geq1$, the operation $\beta\mathcal{P}^{p^r}$
is defined on~$\pi_{2(p^r-1)}E$ giving a map
\[
\beta\mathcal{P}^{p^r}\:\pi_{2(p^r-1)}E \lra \pi_{2(p^{r+1}-1)-1}E.
\]
Moreover the indeterminacy is trivial and the operation
$p\beta\mathcal{P}^{p^r}$ is trivial.
\end{thm}
\begin{proof}
We will assume that all spectra are localised at~$p$. Recall
that $\alpha_1\in\pi_{2p-3}S$ is a non-zero $p$-primary stable
homotopy element of lowest positive degree.

Using the results and notation of~\cite{LNM1176}*{proposition~V.1.5},
the fact that this operation is defined on $\pi_{2(p^r-1)}E$
this follows since
\[
\psi\biggl(2p^r(p-1)-1-2(p^r-1)(p-1)\biggr) = \psi(2(p-1)-1)
      = \left\lfloor\frac{2(p-1)-1}{2(p-1)}\right\rfloor = 0.
\]
For triviality of the indeterminacy, see~\cite{LNM1176}*{table~V.1.1}.

By~\cite{LNM1176}*{theorem~V.1.8}, for each $y\in\pi_{2(p^r-1)}E$
there is an element $\alpha\in\pi_{2(p-1)-2}S$ for which
\[
p(\beta\mathcal{P}^{p^r}y) = \alpha y^p.
\]
But $\alpha\in\pi_{2(p-2)}S=0$, hence $p(\beta\mathcal{P}^{p^r})$
is indeed trivial.
\end{proof}

The next result tells us how this works in the Adams spectral
sequence in good situations.
\begin{lem}\label{lem:PowOps2-1line}
Let $p$ be a prime \\
\emph{(i)} If $p=2$, then under the assumptions of\/
\emph{Theorem~\ref{thm:PowOps2}}, if\/ $w\in\pi_{2^{r+1}-2}E$
is detected in the $1$-line of the Adams spectral sequence
by $W\in\Ext_{\mathcal{A}(2)_*}^{1,2^{r+1}-1}(\F_2,H_*E)$,
then $\mathcal{P}^{2^{r+1}-1}w$ is detected in the
$1$-line by
\[
\mathcal{P}^{2^{r+1}-1}W
       \in\Ext^{1,2^{r+2}-2}_{\mathcal{A}(2)_*}(\F_2,H_*E),
\]
where $\mathcal{P}^{2^{r+1}-1}$ is the algebraic Steenrod
operation of\/~\cites{JPM:StOps,RJM:GpReps&ASS,LNM1176}.
This can be calculated by applying the Dyer-Lashof operation
$Q^{2^{r+1}-1}$ to the element of
$\mathcal{A}(2)_*\otimes H_*E = H_*(H\wedge E)$ representing~$W$. \\
\emph{(ii)} If $p$ is odd, then the assumptions of\/
\emph{Theorem~\ref{thm:PowOps}}, if\/ $w\in\pi_{2(p^r-1)}E$
is detected in the $1$-line of the Adams spectral sequence
by $W\in\Ext_{\mathcal{A}(p)_*}^{1,2p^r-1}(\F_p,H_*E)$, then
$\beta\mathcal{P}^{p^r}w$ is detected in the $1$-line by
\[
\beta\mathcal{P}^{p^r}W
       \in\Ext^{1,2(p^{r+2}-1)}_{\mathcal{A}(p)_*}(\F_p,H_*E),
\]
where $\beta\mathcal{P}^{p^r}$ is the algebraic Steenrod
operation of\/~\cites{JPM:StOps,RJM:GpReps&ASS,LNM1176}. This
can be calculated by applying the Dyer-Lashof operation
$\beta Q^{p^r}$ to the element of
$\mathcal{A}(p)_*\otimes H_*E = H_*(H\wedge E)$ representing~$W$.
\end{lem}
\begin{proof}
This follows from work of Milgram and Bruner~\cites{RJM:GpReps&ASS,LNM1176}.
Note that for $p=2$,~\cite{LNM1176}*{theorem~2.5(i)} should
read
\[
\beta\mathcal{P}^i\:\Ext^{s,t} \lra \Ext^{s+t-i,2t}.
\qedhere
\]
\end{proof}

\section{Outline of a construction}\label{sec:Construction}

In this section and later ones, we will always be working
with (connective) $p$-local spectra for some prime~$p$.
When referring to cells, finite type conditions, etc, we
will always mean in that context.

Starting with the $p$-local sphere $S$, we will construct
a sequence of commutative $S$-algebras
\begin{equation}\label{eqn:SequenceRn}
S=R_0 \lra R_1 \lra \cdots \lra R_{n-1} \lra R_n \lra \cdots,
\end{equation}
where $R_n$ is obtained from $R_{n-1}$ by attaching a single
$E_\infty$ cell of dimension $2(p^n-1)$. The rational homotopy
of the colimit $\ds R_\infty=\colim_n R_n$ is
\[
\Q\otimes \pi_*R_\infty = \Q[u_n:n\geq1],
\]
where $u_n\in\pi_{2(p^n-1)}R_\infty$ arises in $\pi_{2(p^n-1)}R_n$.
Next we could inductively kill the torsion part of the homotopy
of~$R$ by non-trivally attaching $E_\infty$ cones on Moore
spectra, thus we do not change the rational homotopy. Then
we obtain a commutative $R_\infty$-algebra $R$ for which
\[
\Q\otimes \pi_*R = \Q\otimes \pi_*R = \Q[u_n:n\geq1].
\]

\section{Construction of the $R_n$}\label{sec:Rn}

We begin with the construction of the sequence~\eqref{eqn:SequenceRn}.
We will use the notation $u_0=p$.

Let $n\geq1$. Suppose that a sequence of cofibrations of commutative
$S$-algebras
\[
S=R_0 \lra R_1 \lra \cdots \lra R_{n-1}
\]
exists in which there are compatible homotopy elements $u_r\in\pi_{2(p^r-1)}R_k$
for $0\leq r\geq k$, satisfying
\[
\Q\otimes \pi_*R_k=\Q\otimes \pi_*R_{k-1}[u_k].
\]
Then by Theorem~\ref{thm:PowOps}, assuming that it is not trivial,
the element $\beta\mathcal{P}^{p^{n-1}}u_{n-1}$ is of order~$p$;
we let $f_n\:S^{2p^n-3}\lra R_{n-1}$ be a representative of this
homotopy class. Thus as in~\cites{EKMM,TAQ} we can form the pushout
diagram of commutative $S$-algebras
\[
\xymatrix{
\mathbb{P}S^{2p^n-3} \ar[r]^{\ph{abc}\tilde f_n}\ar[d] & R_{n-1}\ar[d] \\
\mathbb{P}D^{2p^n-2} \ar[r] & R_n
}
\]
in which $\tilde f_n$ is the extension of $f_n$ to a map from the
free commutative $S$-algebra $\mathbb{P}S^{2p^n-3}$. We remark that
we can work equally well with commutative $R_{n-1}$-algebras and
define $R_n$ using the pushout diagram
\[
\xymatrix{
\mathbb{P}_{R_{n-1}}S^{2p^n-3} \ar[r]^{\ph{abcd}\tilde f_n}\ar[d]\ar@{}[rd]
      & R_{n-1}\ar[d] \\
\mathbb{P}_{R_{n-1}}D^{2p^n-2} \ar[r] & R_n
}
\]
and we will make use of both viewpoints. We also have
\[
R_n \iso R_{n-1} \wedge_{\mathbb{P} S^{2p^n-3}}\mathbb{P} D^{2p^n-2}
\iso R_{n-1} \wedge_{\mathbb{P}_{R_{n-1}}S^{2p^n-3}}\mathbb{P}_{R_{n-1}}D^{2p^n-2}
\]

Since $f_n$ has order~$p$, there is a commutative diagram
of $R_{n-1}$-modules
\[
\xymatrix{
& & S^{2p^n-2}\ar[rd]^{p}\ar@{.>}[d]\ar@/_25pt/@{-->}[dd] & \\
S^{2p^n-3}\ar[r]^{f_n} & R_{n-1} \ar[r] & C_{f_n}\ar[r]\ar[d] & S^{2p^n-2} \\
 & & R_n &
}
\]
in which the dashed arrow provides a homotopy class
$u_n\in\pi_{2p^n-2}R_n$.

There is a Kunneth spectral sequence~\cite{EKMM} of
the form
\[
\mathrm{E}^2_{r,s} =
\Tor^{\Q\otimes\pi_*R_{n-1}[w_{2p^n-3}]}_{s,t}(\Q\otimes\pi_*R_{n-1},\Q\otimes\pi_*R_{n-1})
                                \Lra \Q\otimes \pi_{s+t}R_n,
\]
where
\[
\Q\otimes\pi_*R_{n-1}[w_{2p^n-3}] =
           \Q\otimes\pi_*\mathbb{P}_{R_{n-1}}S^{2p^n-3}
\]
is an exterior algebra, so
\[
\mathrm{E}^2_{r,s} = \Q\otimes\pi_*R_{n-1}[U_n]
\]
with generator $U_n$ of bidegree $(1,2p^n-3)$. Thus the spectral
sequence collapses and we easily obtain
\[
\Q\otimes\pi_*R_n = \Q\otimes\pi_*R_{n-1}[u_n].
\]

We still need to verify the following key result.
\begin{lem}\label{lem:Rn}
The element $\beta\mathcal{P}^{p^{n-1}}u_{n-1}\in\pi_{2p^n-3}R_{n-1}$
is non-zero and has order~$p$. Furthermore, the
mod~$p$ Hurewicz image of\/ $u_n$ is trivial.
\end{lem}

Passing to the limit, we see that since each morphism
$R_{n-1}\lra R_n$ is a cofibration,
\[
R_\infty = \hocolim_n R_n
\]
and
\[
\pi_*R_\infty = \colim_n \pi_*R_n.
\]
Working rationally this gives
\[
\Q\otimes\pi_*R_\infty = \Q\otimes\pi_*S[u_n:n\geq1]
                       = \Q[u_n:n\geq1].
\]

\section{Killing the torsion}\label{sec:Killtorsion}

The homotopy of the commutative $S$-algebra $R_\infty$ has finite
type and $R$ is a CW commutative $S$-algebra with one $E_\infty$
cell in each degree of the form $2(p^n-1)$ with $n\geq1$.

Now we proceed to kill the torsion in $\pi_*R_\infty$ by induction
on degree. Let $R^0=R_\infty$. Suppose that we have constructed
$R^0\lra R^{m-1}$ so that $\pi_kR^{m-1}$ is torsion free for $k\leq m-2$
and the natural map induces an isomorphism
\[
\Q\otimes\pi_*R^0 \iso \Q\otimes\pi_*R^{m-1}.
\]
Now following~\cites{AJB&JPM,TAQ} we attach $m$-cells minimally
to kill the torsion of $\pi_{m-1}R^{m-1}$. In fact, following
a suggestion of Tyler Lawson, we can do slightly more: factoring
the attaching maps through Moore spectra of the form
$S^{m-1}\cup_{p^r}D^m$, we can define $R^m$ using the pushout
diagram
\[
\xymatrix{
\mathbb{P}(\bigvee_i S^{m-1}\cup_{p^{r_i}}D^m) \ar[r]\ar[d]
    & R^{m-1}\ar[d] \\
\mathbb{P}(\bigvee_i C(S^{m-1}\cup_{p^{r_i}}D^m))\ar[r] & R^m
}
\]
and so we have
\[
\Q\otimes\pi_*R^0 \iso \Q\otimes\pi_*R^m.
\]

Continuing in this way, we obtain a sequence of cofibrations
\[
R_\infty=R^0 \lra R^1\lra \cdots \lra R^{m-1} \lra R^m \lra \cdots
\]
whose limit is
\[
R = \colim_m R^m = \hocolim_m R^m.
\]
Furthermore, the natural map $R_\infty\lra R$ induces
an epimorphism
\[
\pi_*R_\infty\lra\pi_*R
\]
and a rational isomorphism
\[
\Q\otimes\pi_*R\lra \Q\otimes\pi_*R_\infty = \Q[u_n:n\geq1].
\]

\section{Some recursive formulae}\label{sec:Formulae}

We give the odd primary case first, the $2$-primary case is
similar.

\subsection*{The case $p>2$}

Let $p$ be an odd prime and assume that all spectra are $p$-local.
Starting with $R_0=S$, the $p$-local sphere, we will inductively
assume that there is a sequence of $E_\infty$ ring spectra
\[
S=R_0 \lra R_1 \lra \cdots \lra R_{n-1}
\]
so that the following hold: \\
(A) for $1\leq r\leq n$ there are homotopy classes
$\alpha_{[r]}\in\pi_{2(p^r-1)-1}R_{r-1}$ of order~$p$, and
homology classes $z_r\in H_{2(p^r-1)}(R_{r};\F_p)$; \\
(B) the $\mathcal{A}(p)_*$-coaction is given by
\begin{equation}\label{eq:zr-coaction}
\psi(z_r) =
   1\otimes z_r + \zeta_1\otimes z_{r-1}^p
    + \zeta_2\otimes z_{r-2}^{p^2}
    + \cdots + \zeta_{r-1}\otimes z_{1}^{p^{r-1}}
    + \zeta_r\otimes 1,
\end{equation}
where we identify $z_i\in H_*(R_{r};\F_p)$ with the image
of $z_i\in H_*(R_{i};\F_p)$ under the induced homomorphism
$H_*(R_{i};\F_p)\lra H_*(R_{r};\F_p)$ whenever $i<r$; \\
(C) $\alpha_{[r]}$ is detected in filtration $1$ the Adams
spectral sequence by the class with cobar representative
\begin{equation}\label{eq:alpha[r]}
   \zeta_1\otimes z_{r-1}^p
    + \zeta_2\otimes z_{r-2}^{p^2}
    + \cdots + \zeta_{r-1}\otimes z_{1}^{p^{r-1}}
    + \zeta_r\otimes 1,
\end{equation}
and $\alpha_{[1]}=\alpha_1\pi_{2p-3}S$.

Given this data, we construct the morphism of $E_\infty$
ring spectra $R_{n-1}\lra R_n$ as follows.

Choose a representative $f_n\:S^{2p^n-3}\lra R_{n-1}$ for
$\alpha_{[n]}$. Attach an $E_\infty$ cone to $R_{n-1}$ by
forming the pushout $R_{n-1}/\!/\alpha_{[n]}$ in the diagram
\[
\xymatrix{
\mathbb{P}S^{2p^n-3} \ar[r]^{\ph{abc}\tilde f_n}\ar[d] & R_{n-1}\ar[d] \\
\mathbb{P}D^{2p^n-2} \ar[r] & R_{n-1}/\!/\alpha_{[n]}
}
\]
and set $R_n=R_{n-1}/\!/\alpha_{[n]}$. Since $\alpha_{[n]}$
has order~$p$, there is a commutative diagram of $S$-modules
\[
\xymatrix{
& & S^{2p^n-2}\ar[rd]^{p}\ar@{.>}[d]\ar@/_25pt/@{-->}[dd] & \\
S^{2p^n-3}\ar[r]^{f_n} & R_{n-1} \ar[r] & C_{f_n}\ar[r]\ar[d] & S^{2p^n-2} \\
 & & R_n &
}
\]
in which the dashed arrow provides a homotopy class
$u_n\in\pi_{2p^n-2}R_n$. The homology class $z_n$ is
represented by the image of the ordinary cell attached
to form the mapping cone $C_{f_n}$.
\begin{lem}\label{lem:Toda&cobar}
The homotopy class $u_n$ lies in the Toda bracket
$\<p,\alpha_{[n]},1\>\subseteq \pi_{2(p^n-1)}R_n$, and
in the Adams spectral sequence it has filtration\/ $1$
and cobar representative
\[
\bar\tau_0\otimes z_n +\bar\tau_1\otimes z_{n-1}^p
    + \bar\tau_2\otimes z_{n-2}^{p^2}
    + \cdots + \bar\tau_{n-1}\otimes z_{1}^{p^{n-1}},
\]
where $\bar\tau_j$ denotes the conjugate of the exterior
generator $\tau_j\in\mathcal{A}(p)_{2p^j-1}$.
\end{lem}
\begin{proof}
This Toda bracket should be interpreted in the sense of
modules over $R_{n-1}$. Thus the first two variables are
in $\pi_*R_{n-1}$ while the last is in $\pi_*R_n$ viewed
as a module over $\pi_*R_{n-1}$.

Now in the Adams $\mathrm{E}_2$-term, we have the relation
\[
h_0[\zeta_1\otimes z_{n-1}^p
    + \zeta_2\otimes z_{n-2}^{p^2}
    + \cdots + \zeta_{n-1}\otimes z_{1}^{p^{n-1}}
    + \zeta_n\otimes 1] = 0
\]
since $p\alpha_{[n]}=0$, and using~\eqref{eq:zr-coaction}
we obtain
\begin{align*}
d_1\biggl(  \sum_{1\leq s\leq n}\bar\tau_s\otimes z_{n-s}^{p^s} \biggr)
=& \sum_{1\leq s\leq n} 1\otimes\bar\tau_s\otimes z_{n-s}^{p^s} \\
&\quad - \biggl(
   \sum_{\substack{1\leq s\leq n\\ 0\leq j\leq s}}
     \bar\tau_j\otimes\zeta_{s-j}^{p^j}\otimes z_{n-s}^{p^s}
     + \sum_{1\leq s\leq n} 1\otimes\bar\tau_s\otimes z_{n-s}^{p^s}
    \biggr) \\
&\ph{abcdefgj}
     + \biggl(\sum_{1\leq s\leq n} \bar\tau_s\otimes 1\otimes z_{n-s}^{p^s}
     + \sum_{\substack{1\leq s\leq n\\ 1\leq k\leq n-s}}
            \bar\tau_s\otimes\zeta_k^{p^s} \otimes z_{n-s-k}^{p^{s+k}}\biggr) \\
=&  - \sum_{\substack{1\leq s\leq n\\ 0\leq j\leq s-1}}
     \bar\tau_j\otimes\zeta_{s-j}^{p^j}\otimes z_{n-s}^{p^s}
   + \sum_{\substack{1\leq s\leq n\\ 1\leq k\leq n-s}}
        \bar\tau_s\otimes\zeta_k^{p^s} \otimes z_{n-s-k}^{p^{s+k}} \\
=&  - \sum_{1\leq s\leq n} \bar\tau_0\otimes\zeta_{s}\otimes z_{n-s}^{p^s} \\
&\ph{abcdefgh}  - \biggl(\sum_{\substack{1\leq s\leq n\\ 1\leq j\leq s-1}}
     \bar\tau_j\otimes\zeta_{s-j}^{p^j}\otimes z_{n-s}^{p^s}
   - \sum_{\substack{1\leq s\leq n\\ 1\leq k\leq n-s}}
        \bar\tau_s\otimes\zeta_k^{p^s} \otimes z_{n-s-k}^{p^{s+k}}\biggr) \\
=&  \sum_{1\leq s\leq n} -\bar\tau_0\otimes\zeta_{s}\otimes z_{n-s}^{p^s}.
\end{align*}
We also have
\begin{align*}
d_1(-z_n) =&  - 1\otimes z_n
  + \biggl(\sum_{1\leq s\leq n}\zeta_t\otimes z_{n-t}^{p^t}
                                          + 1\otimes z_n \biggr) \\
=&  \sum_{1\leq s\leq n}\zeta_s\otimes z_{n-s}^{p^s}.
\end{align*}
Therefore we have
\[
\biggl[\sum_{0\leq s\leq n}\bar\tau_s\otimes z_{n-s}^{p^s}\biggr] =
\biggl[\bar\tau_0\otimes z_n + \sum_{1\leq s\leq n}\bar\tau_s\otimes z_{n-s}^{p^s}\biggr]
\in \<h_0,\biggl[\sum_{1\leq s\leq n}\bar\tau_s\otimes z_{n-s}^{p^s}\biggr],1\>.
\]
So modulo higher Adams filtration, the Toda bracket $\<p,\alpha_{[n]},1\>$
is represented in the Adams spectral sequence by
\[
 \<h_0,\biggl[\sum_{1\leq s\leq n}\bar\tau_s\otimes z_{n-s}^{p^s}\biggr],1\>.
\qedhere
\]
\end{proof}

Note that Lemma~\ref{lem:PowOps2-1line}(ii) gives
\begin{align*}
\beta\mathcal{P}^{p^n}
\biggl[\sum_{0\leq s\leq n}\bar\tau_s\otimes z_{n-s}^{p^s}\biggr]
&= \biggl[\beta Q^{p^n}\sum_{0\leq s\leq n}\bar\tau_s\otimes z_{n-s}^{p^s}\biggr] \\
&= \biggl[\sum_{0\leq s\leq n}\beta Q^{p^s}(\bar\tau_s)\otimes Q^{p^n-p^s}(z_{n-s}^{p^s})\biggr] \\
&= \biggl[\sum_{0\leq s\leq n}\beta\bar\tau_{s+1}\otimes(Q^{p^{n-s}-1}z_{n-s})^{p^s}\biggr] \\
\intertext{\hfill (by \cite{LNM1176}*{theorem~III.2.3})}
&= \biggl[\sum_{0\leq s\leq n}\beta\bar\tau_{s+1}\otimes(z_{n-s}^p)^{p^s}\biggr] \\
&= \biggl[\sum_{0\leq s\leq n}\zeta_{s+1}\otimes z_{n-s}^{p^{s+1}}\biggr] \\
\intertext{\hfill (by \cite{LNM1176}*{theorem~III.2.3} again)}
&= \biggl[\sum_{1\leq s\leq n}\zeta_s\otimes z_{n+1-s}^{p^s}\biggr]  \\
&= \alpha_{[n+1]}.
\end{align*}

\subsection*{The case $p=2$}

With similar notation to that for odd primes, we have

\begin{lem}\label{lem:Toda&cobar2}
The element $u_n$ lies in the Toda bracket
$\<2,w_{n-1},1\>\subseteq \pi_{2^{n+1}-2}R(n)$, and
in the Adams spectral sequence it has filtration\/~$1$
with cobar representative
\[
\zeta_1\otimes s_n +\zeta_2\otimes s_{n-1}^2
    + \zeta_3\otimes s_{n-2}^{2^2}
    + \cdots + \zeta_n\otimes s_{1}^{2^{n-1}}
    + \zeta_{n+1}\otimes1,
\]
where $\zeta_j$ denotes the conjugate of the Milnor
generator generator $\xi_j\in\mathcal{A}(2)_{2^j-1}$.
\end{lem}

\section{The map to $H\mathbb{F}_p$} \label{sec:R->HF_p}

There is a morphism of commutative $S$-algebras $R_\infty\lra MU_{(p)}$,
and composing this with the Quilen morphism of ring spectra
$\epsilon\:MU_{(p)}\lra BP$ we obtain morphisms of ring spectra

\smallskip
\[
\xymatrix{
R_\infty \ar[r]\ar@/^15pt/[rr] &  BP\ar[r] & H\F_p
}
\]
and we would like to understand their induced maps
in homotopy and homology.

\begin{lem}\label{lem:BP-gens}
Let $p$ be a prime. \\
\emph{(i)} If $p$ is odd, suppose that $s_1,s_2,s_2,\ldots$
is a sequence of elements $s_n\in\mathcal{A}(p)_{2(p^n-1)}$
with coproducts
\[
\psi(s_n) =
\zeta_n\otimes1 + \zeta_{n-1}\otimes s_1^{p^{n-1}}
+ \cdots + \zeta_1\otimes s_{n-1}^{p} + 1\otimes s_n.
\]
Then $s_n = \zeta_n$. \\
\emph{(ii)} If $p=2$, suppose that $s_1,s_2,s_2,\ldots$
is a sequence of elements $s_n\in\mathcal{A}(p)_{2^n-1}$
with coproducts
\[
\psi(s_n) =
\zeta_n\otimes1 + \zeta_{n-1}\otimes s_1^{2^{n-1}}
+ \cdots + \zeta_1\otimes s_{n-1}^2 + 1\otimes s_n.
\]
Then $s_n = \zeta_n$.
\end{lem}
\begin{proof}
We recall that there are no non-trivial coaction primitives
in positive degrees, \ie, viewing $\mathcal{A}(p)_*$ as a
left $\mathcal{A}(p)_*$-comodule, a standard change of rings
isomorphism gives
\[
\Ext^{0,*}_{\mathcal{A}(p)_*}(\F_p,\mathcal{A}(p)_*)
\iso \Ext^{0,*}_{\F_p}(\F_p,\F_p) = \F_p.
\]
(i) For $n=1$, we have
\begin{align*}
\psi(s_1-\zeta_1) &=
(\zeta_1\otimes1+1\otimes s_1)-(\zeta_1\otimes1+1\otimes \zeta_1) \\
&= 1\otimes(s_1-\zeta_1).
\end{align*}
So $s_1=\zeta_1$.

Now suppose that for $k<n$, $s_k=\zeta_k$. Then
\begin{align*}
\psi(s_n-\zeta_n) &=
\sum_{0\leq j\leq n}\zeta_j\otimes s_{n-j}^{p^j}
- \sum_{0\leq j\leq n}\zeta_j\otimes \zeta_{n-j}^{p^j} \\
&= 1\otimes(s_n-\zeta_n),
\end{align*}
so we have $s_n=\zeta_n$. By induction this holds
for all $n$.

The proof of (ii) is similar.
\end{proof}
\begin{rem}\label{rem:BP-gens}
Since $H_*(BP;\F_p)$ can be identified with a subalgebra
of $\mathcal{A}(p)_*$, we can also characterize a family
of polynomial generators $t_n\in H_{2(p^n-1)}(BP;\F_p)$
by the coaction formulae
\[
\psi(t_n) =
\begin{cases}
\ds\sum_{0\leq j\leq n} \zeta_j\otimes t_{n-j}^{p^j} & \text{if $p$ is odd}, \\
\ds\sum_{0\leq j\leq n} \zeta_j^2\otimes t_{n-j}^{2^j} & \text{if $p=2$}.
\end{cases}
\]
\end{rem}

\begin{thm}\label{thm:Rinfty->BP}
The morphism of ring spectra $R_\infty\lra BP$ induces
epimorphisms in $\pi_*(-)$, $H_*(-;\Z_{(p)})$ and $H_*(-;\F_p)$.
\end{thm}
\begin{proof}
We indicate two rather different proofs. \\
\textbf{First proof:}
The morphism of ring spectra $R_\infty\lra BP\lra H\F_p$
induces a homorphism in homology sending the elements
$z_n$ to elements $s_n\in\mathcal{A}(p)_*$ for which
Lemma~\ref{lem:BP-gens} applies. By Remark~\ref{rem:BP-gens},
this means that $z_n\mapsto \zeta_n$ if $p$ is odd, and
$z_n\mapsto \zeta_n^2$ if $p=2$. \\
\textbf{Second proof:}
First assume that $p$ is odd. Consider the morphism of ring
spectra $R_1\lra R_\infty\lra BP$. The $z_1\in H_{2(p-1)}(R_1;\F_p)$
maps to an element $t\in H_{2(p-1)}(BP;\F_p)$ with
$\mathcal{A}_*$-coaction
\[
\psi(t) = \zeta_1\otimes 1 + 1\otimes t.
\]
The only such element is $t_1$.

The homomorphism $H_*(BP;\F_p)\lra H_*(H\F_p;\F_p)=\mathcal{A}(p)_*$.
Also $MU\lra H\F_p$ is a morphism of commutative $S$-algebras whose
image is $H_*(BP;\F_p)\subseteq H_*(H\F_p;\F_p)$. Therefore the
action of the Dyer-Lashof operations on $H_*(H\F_p;\F_p)$ restricts
to $H_*(BP;\F_p)$. Now $t_r$ maps to $\zeta_r$, so we can determine
the Dyer-Lashof action using~\cite{LNM1176}*{theorem~III.2.3}. Then
\[
Q^{p^{s-1}}\cdots Q^{p^2}Q^p\zeta_1 = \zeta_s,
\]
hence
\[
Q^{p^{s-1}}\cdots Q^{p^2}Q^pt_1 = t_s.
\]
Thus the element $Q^{p^{s-1}}\cdots Q^{p^2}Q^pz_1\in H_*(R_1;\F_p)$
maps to $t_s\in H_*(BP;\F_p)$. Since
\[
H_*(BP;\F_p) = \F_p[t_s:s\geq1],
\]
we see that $H_*(R_1;\F_p)\lra H_*(BP;\F_p)$ is epic, hence
so is $H_*(R_\infty)\lra H_*(BP;\F_p)$.

In fact the $z_s$ all lift to elements of $H_*(R_\infty;\Z_{(p)})$
and it easily follows that $H_*(R_\infty;\Z_{(p)})\lra H_*(BP;\Z_{(p)})$
is epic.

For $p=2$, the arguments are similar, but with $\zeta_s^2$ in
place of $\zeta_s$, and $Q^{2^{r+1}}$ in place of $Q^{p^r}$
throughout.

To show that the induced homomorphism $\pi_*R_\infty\lra\pi_*BP$
in homotopy is epic, we need to verify that a family of polynomial
generators for $\pi_*BP$ is in the image. When~$p$ is odd,
Lemma~\ref{lem:Toda&cobar} together with the above discussion,
shows that in the Adams spectral sequence for $\pi_*BP$, $u_n$
maps to an element represented by
\[
\biggl[\bar\tau_0\otimes t_n +\bar\tau_1\otimes t_{n-1}^p
    + \bar\tau_2\otimes t_{n-2}^{p^2}
    + \cdots + \bar\tau_{n-1}\otimes t_{1}^{p^{n-1}}\biggr]
    \in \Ext_{\mathcal{A}(p)_*}^{1,2p^n-1}(\F_p,H_*(BP;\F_p))
\]
which correspond to a homotopy element with Hurewicz image
in $H_*(BP;\Z_{(p)})$ of the form
\[
pt_n \pmod{p,\mathrm{decomposables}}.
\]
By Milnor's criterion, this is a polynomial generator.

The argument for $p=2$ is similar, with $u_n$ mapping to an
element having cobar representative
\[
[\zeta_1\otimes t_n + \zeta_2\otimes t_{n-1}^2
    + \zeta_3\otimes t_{n-2}^{2^2} + \cdots + \zeta_{n+1}\otimes1]
   \in \Ext_{\mathcal{A}(2)_*}^{1,2^n-1}(\F_2,H_*(BP;\F_2)).
\qedhere
\]
\end{proof}

\section{Relationship to $BP$}\label{sec:BP?}

We start with an easy lemma. For an abelian group $G$,
we write $\tors G$ for the torsion subgroup.
\begin{lem}\label{lem:Moorespaceconing}
Let $Y\xrightarrow{g} Z$ be a fibration of $p$-local
spectra and let
\[
\xymatrix{
X \ar[r]^{f}\ar[d] & Y\ar[d]^{g}  \\
{*} \ar[r] & Z \ar@{}[ul]|{\PB}
    %{\text{\Large$\lrcorner$}}
}
\]
be a pullback square. Assume that the following hold:
\begin{itemize}
\item
$f_*\:\pi_*(X)\lra\pi_*(Y)$ is monic;
\item
$\tors\pi_*(Z)=0$;
\item
$\tors\pi_*(X)=\pi_*(X)$.
\end{itemize}
Suppose that $\alpha\in\tors\pi_m(Y)$ is non-zero and
has order~$p^e$. Then there is a map
\[
u\:S^m\cup_{p^e}D^{m+1}\lra X
\]
for which the composition
\[
\xymatrix{
S^m \ar[r]^(.3){\inc}\ar@/_18pt/[rrr] &
S^m\cup_{p^e}D^{m+1}\ar[r]^(.7){u} &
X\ar[r]^{f} & Y
}
\]
represents $\alpha$.
\end{lem}
\begin{proof}
By assumption, $f_*$ induces an isomorphism
\[
f_*\:\pi_*(X) \xrightarrow{\;\iso\;}\tors\pi_*(Y).
\]
hence there is a unique element $\alpha'\in\pi_m(X)$
for which $f_*(\alpha')=\alpha$ and the order of
$\alpha'$ is also $p^e$. A representative of $\alpha'$
must factor through $S^m\cup_{p^e}D^{m+1}$,
\bigskip
\[
\xymatrix{
S^m\ar[r]\ar@/^20pt/[rr]\ar@/_10pt/@{-->}[drr]
  & S^m\cup_{p^e}D^{m+1}\ar[r]_(.7){u} & X\ar[d]^{f} \\
  && Y
}
\]
showing that the desired $u$ exists, and the dashed
arrow represents $\alpha$.
\end{proof}
\begin{cor}\label{cor:Moorespaceconing}
The map $g$ factors through the mapping cone of $fu$.
\[
\xymatrix{
S^m\cup_{p^e}D^{m+1}\ar[r]^(.7){fu} & Y\ar[r]\ar[d]_{g}
    & \mathrm{C}_{fu} \ar@{-->}[dl] \\
& Z &
}
\]
\end{cor}
\begin{proof}
This follows from the commutative diagram
\[
\xymatrix{
S^m\cup_{p^e}D^{m+1}\ar[r]^(.7){u}\ar@{=}[d] & Y\ar[r]\ar[d]_{f}
    & \mathrm{C}_{u} \ar[d] \\
S^m\cup_{p^e}D^{m+1}\ar[r]^(.7){fu} & Y\ar[r]\ar[d]_{g}
    & \mathrm{C}_{fu} \ar@{-->}[dl] \\
& Z &
}
\]
in which $gf$ is the trivial map and the dashed arrow
is obtained by mapping the cone trivially.
\end{proof}

\begin{thm}\label{thm:R=BP}
Let $p$ be a prime. If $BP$ admits an $E_\infty$ structure
then there is a weak equivalence of commutative $S$-algebras
$R\xrightarrow{\;\sim\;}BP$.
\end{thm}
\begin{proof}
Since $\pi_*BP$ is torsion-free, the inductive construction
of $R_n$ from $R_{n-1}$ gives morphisms of commutative
$S$-algebras
\[
\xymatrix{
R_{n-1}\ar[rr]\ar[dr] && R_n\ar@{.>}[dl] \\
 & BP &
}
\]
and passing to the colimit we obtain a morphism $R_\infty\lra BP$.
By Theorem~\ref{thm:Rinfty->BP}, $\pi_*R_\infty\lra\pi_*BP$ is an
epimorphism and on tensoring with $\Q$ it becomes an isomorphism.

On replacing $R_\infty\lra BP$ with a fibration of commutative
$S$-algebras $T^0\lra BP$ with fibre $J_0$, we are in the situation
of Lemma~\ref{lem:Moorespaceconing}. Now we can inductively adjoin
cones on wedges of Moore spectra $S^m\cup_{p^e}D^{m+1}$ where
$m\geq1$ to form morphisms of $E_\infty$ ring spectra~$T^{m-1}\lra T^m$.
At each stage Corollary~\ref{cor:Moorespaceconing} shows that we
can extend to a diagram of morphisms
\[
\xymatrix{
T^0 \ar[r]\ar[drrr] & T^1 \ar[r]\ar[drr] & \cdots\cdots \ar[r]
     & T^{m-1}\ar[r]\ar[d] & T^m\ar[r]\ar[dl] &\cdots \\
 &&&BP&&
}
\]
and the homotopy colimit $\hocolim_m T^m$ is easily
seen to admit a weak equivalence to
\[
\hocolim_m T^m\xrightarrow{\;\sim\;} BP.
\]
As $R\sim\hocolim_m T^m$, this shows that $R\sim BP$.
\end{proof}

As defined, it is not clear if $R$ is a minimal atomic commutative
$S$-algebra; however, by construction, $R_\infty$ is nuclear
and hence is minimal atomic according to results of~\cite{TAQ}.
We can produce a core $R^c\lra R$, \ie, a morphism of commutative
$S$-algebras with $R^c$ nuclear and which induces a monomorphism
on $\pi_*(-)$. In particular, $\pi_*(R^c)$ is torsion-free.
\begin{lem}\label{lem:vn}
Let $A$ be a connective $p$-local commutative $S$-algebra for
which $\pi_*(A)$ is torsion-free. Then there is a morphism of
commutative $S$-algebras $R_\infty\lra A$. In particular, the
natural morphism $R_\infty\lra R$ admits a factorisation
through any core $R^c\lra R$ for $R$.
\[
\xymatrix{
R_\infty \ar[r]\ar@/^12pt/[rr] & R^c \ar[r] & R
}
\]
\end{lem}
\begin{proof}
Since our cellular construction of $R_\infty$ involves attaching
$E_\infty$ cells to kill torsion elements in homotopy, it is
straightforward to see that at each stage we can extend the unit
$S\lra A$, in the limit this gives a morphism $R_\infty\lra A$.
\end{proof}

As $R$ and more generally any core $R^c$ have torsion-free homotopy
concentrated in even degrees, standard arguments of~\cite{JFA:Blue}
show that there are morphisms of ring spectra $BP\lra R$ and
$BP\lra R^c$ associated with complex orientations with $p$-typical
formal group laws. Our earlier arguments show that these are rational
weak equivalences. Of course we have not shown that $BP\sim R$ even
as (ring) spectra. One way to prove this would be to produce any map
of spectra $R\lra BP$ that is an equivalence on the bottom cell, for
then the composition $BP\lra R\lra BP$ would be a weak equivalence,
therefore so would each of the maps $BP\lra R$ and $R\lra BP$. It
is tempting to conjecture that $R$ (or equivalently $R^c$) is always
weakly equivalent to $BP$, but we have no hard evidence for this
beyond what we have described above.

\appendix
\section{Toda brackets and Massey products}\label{App:Toda}

For the sake of completeness, we describe the kind of Toda
brackets and Massey products we use. Details of this material
can be developed in the spirit of the exposition of Toda
brackets by Whitehead~\cite{GW:Advances}.

\subsection*{Toda brackets in the homotopy of $R$-modules}

We will work with (left) $S$-modules in the sense
of~\cite{EKMM}. We will usually omit $S$ from notation,
for example $\wedge$ will denote $\wedge_S$ and so on.

Let $R$ be a commutative $S$-algebra and let $M$ be a
left $R$-module. We will require Toda brackets of the
following form. Let $\alpha\in\pi_aR$, $\beta\in\pi_bR$
and let $\gamma\in\pi_cM$, and suppose that
\[
\alpha\beta = 0 = \beta\gamma.
\]
Choosing representatives $f\:S^a\lra R$, $g\:S^b\lra R$,
$h\:S^c\lra M$, the maps

\smallskip
\[
\xymatrix{
S^{a+b}\ar[r]\ar@/^18pt/[rrr]^{fg} & S^a\wedge S^b\ar[r]_{f\wedge g}
                        & R\wedge R \ar[r]& R  &
S^{b+c}\ar[r]\ar@/^18pt/[rrr]^{gh} & S^b\wedge S^c\ar[r]_{g\wedge h}
                        & R\wedge M \ar[r]& M
}
\]
are null homotopic. Now choosing explicit null homotopies
\[
k\:D^{a+b+1} \lra R,\quad \ell\: D^{b+c+1}\lra M,
\]
we obtain maps

\smallskip
\[
\xymatrix{
D^{a+b+1}\wedge S^c\ar[rr]_(.6){k\wedge h}\ar@/^18pt/[rrr]^{kh} && R\wedge M\ar[r]
                       & M  &
S^a\wedge D^{b+c+1}\ar[rr]_(.6){f\wedge\ell}\ar@/^18pt/[rrr]^{f\ell} && S^b\wedge S^c\ar[r]
                        & M
}
\]
which agree on the boundary
$S^{a+b+c}\doteq S^{a+b}\wedge S^c\doteq S^{a}\wedge S^{b+c}$.
Therefore we obtain a map $S^{a+b+c+1}\lra M$ in the usual way
representing the bracket $\<\alpha,\beta,\gamma\>$.

\subsection*{Recollections on Massey products}

We follow the sign conventions of~\cite{SOK:Book}*{section~5.4}.

Let $(A,d)$ be a dga where $A=A^*$ is $\Z$-graded. If $W\in A^*$
is a homogeneous element, we set
\[
\bar{W} = (-1)^{1+\deg W}W.
\]
Suppose that $X,Y,Z\in A^*$ are homogeneous elements which
are cycles so that the Massey product $\<[X],[Y],[Z]\>$ is
defined, \ie, $[X][Y]=0=[Y][Z]$. Choose $U,V\in A^*$ so
that
\[
d(U) = \bar{X}Y, \quad d(V) = \bar{Y}Z.
\]
Then $d(\bar{X}V + \bar{U}Z) = 0$ and
\[
[\bar{X}V + \bar{U}Z] \in \<[X],[Y],[Z]\>\subseteq H^*(A,d).
\]
The indeterminacy is the subset
\[
[X]\.H^*(A,d) + \.[Z]\.H^*(A,d) \subseteq H^*(A,d).
\]

\end{document}